\documentclass[11pt]{amsart}

\usepackage{amssymb}
\usepackage[all]{xy}
\usepackage[colorlinks=true, urlcolor=rltblue, citecolor=drkgreen, linkcolor=drkred] {hyperref}
\usepackage{color}
\usepackage{pdfsync}
\definecolor{rltblue}{rgb}{0,0,0.4}
\definecolor{drkgreen}{rgb}{0,0.4,0}
\definecolor{drkred}{rgb}{0.5,0,0}
\newtheorem{thm}{Theorem}[section]
\newtheorem{lemma}[thm]{Lemma}

\newtheorem{theorem}[thm]{Theorem}
\newtheorem{corollary}[thm]{Corollary}

\theoremstyle{definition}
\newtheorem{definition}[thm]{Definition}

\theoremstyle{remark}

\newtheorem{historic}[thm]{Historic Remark}

\theoremstyle{plain}

\newcounter{contenumi}

\def\geqt{\geq_T}

\def\upto{\mathop{\upharpoonright}}

\def\and{\mathrel{\&}}

\def\sminus{\smallsetminus}

\def\isom{\cong}
\def\Si{\Sigma}
\newcommand\rightdate[1]{\footnotetext{  Saved: #1 \\ Compiled: \today}}
\def\A{\mathcal{A}}
\def\B{\mathcal{B}}
\def\C{\mathcal{C}}

\def\om{\omega}
\def\bbar{\bar{b}}

\def\si{\sigma}
\def\b{\beta}
\def\a{\alpha}
\def\b{\beta}

\def\ZZ{{\mathbb Z}}
\def\QQ{{\mathbb Q}}
\def\M{{\mathcal M}}

\def\A{\mathcal A}

\def\X{\mathcal X}
\def\Y{\mathcal Y}
\def\L{\mathcal L}
\def\H{\mathcal H}
\def\xbar{\bar{x}}
\def\cbar{\bar{c}}

\def\abar{\bar{a}}
\def\bbar{\bar{b}}
\def\xbar{\bar{x}}
\def\dbar{\bar{d}}
\def\ctt{{\mathtt c}}
\def\itt{{\mathtt {in}}}
\def\Sic{\Si^\ctt}
\def\Pic{\Pi^\ctt}
\newcommand{\Sico}[1]{\Si^{\ctt,#1}}
\newcommand{\Pico}[1]{\Pi^{\ctt,#1}}
\def\Sii{\Si^\itt}
\def\Pii{\Pi^\itt}

\def\KK{{\mathbb K}}
\def\SS{{\mathbb S}}

\def\RR{{\mathbb R}}

\def\Phihat{\hat{\Phi}}
\def\Lhat{\hat{\L}}
\def\Ahat{\hat{\A}}
\def\Bhat{\hat{\B}}

\def\g{\gamma}
\def\d{\delta}

\def\Si{\Sigma}

\def\X{{\mathcal X}}
\def\S{{\mathcal S}}

\def\om{\omega}

\def\implies{\Rightarrow}

\def\That{\hat{T}}
\def\sihat{\hat{\si}}

\def\Ltil{\tilde{\L}}

\newcommand{\fseq}[1]{2^{\circ #1}}

\title{Classes of structures with no intermediate isomorphism problems}

\author{Antonio Montalb\'an}
\thanks{The author was partially supported by NSF grant DMS-0901169 and the Packard Fellowship.
The author would like to thank Leo Harrington, Julia Knight, Sy Friedman, Russell Miller and Ekaterina Fokina for useful conversations.
This paper was written while the author participated in the Buenos Aires 
Semester in Computability, Complexity and Randomness, 2013.
}

\address{Department of Mathematics\\
University of California, Berkeley\\
 USA}

\begin{document}

\rightdate{September 15, 2013--submitted for publication}
\maketitle

%\tableofcontents

%

\begin{abstract}
We say that a theory $T$ is intermediate under effective reducibility if the isomorphism problems among its computable models is neither hyperarithmetic nor on top under effective reducibility.
We prove that if an infinitary sentence $T$ is uniformly effectively dense, a property we define in the paper, then no extension of it is intermediate, at least when relativized to every oracle on a cone.
As an application we show that no infinitary sentence whose models are all linear orderings is intermediate under effective reducibility relative to every oracle on a cone.
\end{abstract}

%%%%%%%%%%%%%%%%%%%%%%%%%%%%%%%%%%%%%%%%%%%%%%%%%%%%%%%%%%%%%%%%%%%%%%%%%%%%%%%%%%%%%%%%%%%%%%%%%%%%%%%%%%%%%%%%
\section{Introduction}

In this paper, we show a connection between Vaught's conjecture and an intriguing open question about computable structures.
The question we are referring to asks whether every nice theory $T$ (given by a computably infinitary sentence) satisfies what we call the {\em no-intermediate-extension  property}, which essentially means that for every nice extension $\That$ of $T$ (i.e.\  $\That=T\wedge \varphi$ where $\varphi$ is a  computable infinitary sentence), the isomorphism problem among the computable models of $\That$ is either ``simple,''  or as complicated as possible, but is never intermediate. 
By ``simple'' here we mean hyperarithmetic, and by ``as complicated as possible'' we mean universal  among all $\Si^1_1$-equivalences relations on $\om$ under effective reducibility. See Definition \ref{def: no intermediate extension}.
It is already known that if $T$ has this property when relativized to all oracles, then Vaught's conjecture holds among the extensions of $T$ (Becker \cite{Bec}).
The main result of this paper is a partial reversal, showing that the no-intermediate-extension property follows from a strengthening of Vaught's conjecture, which we call the {\em uniform-effective-density property}.

As a bit of evidence that this strengthening is not too strong, we show that the theory of linear orderings has the uniform-effective-density property.
It thus follows that the isomorphism problem among the computable models of any given theory $\That$ extending that of linear orderings, is either hyperarithmetic or as complicated as possible, but never intermediate, at least relative to every oracle on a cone (Theorem \ref{thm: linear no int}). 

As a side result that follows from one of our lemmas, we show that if a nice class of structures is on top under hyperarithmetic reducibility on a cone, then it is already on top under computable reducibility, also on a cone (Theorem \ref{thm: hyp on top}).

Let us now explain all these concepts in more detail and  give some of the background behind them.

\subsection*{The no-intermediate-extension property}
In \cite{FF09}, K.\ Fokina and S.\ Friedman  started to analyze an effective version of the H.\ Friedman and L.\ Stanley \cite{FS89} reducibility among classes of structures.

\begin{definition}
We say that  a class of structures $\KK$ {\em is on top under effective reducibility} if for every $\Si^1_1$ equivalence relation $E$ on $\om$, there is a computable function $f\colon \om\to\om$, mapping numbers to indices for computable structures in $\KK$ such that, for all $i,e\in\om$, 
\[
i\ E\ e \iff\A_{f(i)}\isom\A_{f(e)},
\]
where $\A_n$ is the computable structure coded by the $n$th Turing machine.
\end{definition}

K.\ Fokina, S.\ Friedman, V.\ Harizanov, J.\ Knight, C.\ MaCoy and A.\ Montalb\'an \cite{FFHKMM} proved that the classes of linear orderings, trees, fields, $p$-groups, torsion-free abelian groups, etc.\ are all on top under effective reducibility.
The only examples of classes of structures that we know are not on top under effective reducibility are the ones where the isomorphism problem among computable structures is hyperarithmetic, such as vector spaces, equivalence structures, torsion-free groups of finite rank, etc. 

This behavior is quite different from that of the  Friedman--Stanley reducibility, where they consider all the countable models of a theory (coded by reals), not just the computable ones, and used Borel functions are reducibilities.
There, a class of structures $\KK$ is said to be {\em on top under Borel reducibility} if for every other class of structures $\SS$, axiomatizable by an $L_{\om_1,\om}$-sentence, there is a Borel function mapping structures $\SS$ to structures in $\KK$ preserving isomorphism. 
In that context, no isomorphism problem can be on top among all analytic equivalence relations on the reals.
Friedman and Stanley provided some examples of classes that are on top under Borel reducibility, like linear orderings, trees and fields. 
However, for $p$-groups, which we said they were on top under effective reducibility, Friedman and Stanley showed they are not on top under Borel reducibility, despite the fact that the isomorphism problem is $\Si^1_1$-complete as a subset of $\RR^2$ (which is different than being universal as an equivalence relation).
For torsion-free abelian groups, which we also said they were on top under effective reducibility, it is still open whether they are on top under Borel reducibility, and all that it is known is that their isomorphism problem is $\Si^1_1$-complete as a set of pairs of reals \cite{Hjo02, torsionFree}.

\begin{definition}
Let us call a class of structures, $\KK$, {\em intermediate for effective reducibility} if it is not on top under effective reducibility, but also the isomorphism problem among its computable structures (i.e., the set $\{(i,e)\in\om^2:\A_{i},\A_{e}\in\KK, \A_{i}\isom\A_{e}\}$) in not hyperarithmetic.
\end{definition}

Let us remark that there are natural equivalence relations on $\om$ that are intermediate, as for example the relation of bi-embeddability among computable linear orderings (it is not on top because it has only one non-hyperarithmetic equivalence class, namely the class of $\QQ$.
For more on bi-embeddability of linear orderings see \cite{MonBSL}).
However, we do not know of an example where the equivalence relation is isomorphism on a nice class of structures.

In  \cite{FFHKMM}, they  asked  whether the following statement is true.

\begin{quote} 
No class of structures axiomatizable by a computably infinitary sentence is intermediate under effective reducibility.
\end{quote}

It is often the case in computable-structure theory that the relativized notions behave better than the unrelativized ones, as they avoid ad-hoc counter-examples. 
In this paper, we concentrate on the relativized notions:

\begin{definition}\label{def: no intermediate extension}
An infinitary sentence $T$ is {\em intermediate on a cone} if there exists a $C\in 2^\om$ (the base of the cone), such that relative to every oracle $X\geqt C$, the class of models of $T$ is intermediate for effective reducibility.
We say that $T$ has the {\em no-intermediate-extension property} if no extension $\That$ of $T$ is intermediate on a cone.
By an {\em extension} of $T$ we mean a sentence of the form $T\wedge \varphi$ where $\varphi$ is an infinitary sentence.
\end{definition}

As an application of our results we will show the following theorem.

\begin{theorem} \label{thm: linear no int}(ZFC+PD)
The theory of linear orderings has the no-intermediate-extension property.
\end{theorem}

Because this property is about theories being not intermediate {\em on a cone}, it does not exactly follow from the theorem that if we take a nice extension of the theory of linear orderings, say given by a computably infinitary sentence, then it is not be intermediate for effective reducibility (relative to 0)--but it almost does.
What does follow, however, is that even if there was such an intermediate extension, there cannot be a relativizable proof that it is intermediate.

It was already known that linear orderings satisfy Vaught's conjecture, as proved by Rubin \cite{Rub74} (see also Steel \cite{Ste78}).
In Section \ref{ss: VC LO}, using part of the construction we use for Theorem \ref{thm: linear no int}, we give another proof of that fact.
A connected result worth mentioning is that the extensions of the theory of linear ordering satisfy the Glimm--Effros dichotomy (Gao \cite{Gao01}).

For arbitrary theories, and for one of the implications, we have the following theorem.

\begin{theorem}[\cite{Bec}] \label{thm: becker}
If $T$ has the no-intermediate-extension property, then $T$ {\em satisfies Vaught's conjecture}, in the sense that every extension $\That$ of $T$ has either countably many, or continuum many countable models.
\end{theorem}

The result above was first proved by Becker \cite{Bec}.
Knight and Montalb\'an arrived at the same conclusion roughly at the same time via a very different proof.
They use techniques from computable structure theory, while Becker uses techniques from invariant descriptive set theory.
Both proofs show that if $T$ is a minimal counter-example to Vaught's conjecture, then there is an oracle relative to which there is exactly one computable model of $T$ with a non-hyperarithmetic index set.
To prove this, Knight and Montalb\'an show (using \cite[Lemma 3.3]{MonVC}) that there is an oracle relative to which $T$ has exactly one computable model of high Scott rank, and then modify the oracle to get the index set for this structure to be not hyperarithmetic.

\subsection*{Hyperarithmetic reductions}
The natural effectivization to computable models of the Friedman--Stanley reducibility would be to consider hyperarithmetic reductions instead of computable reductions.
We say that a class $\KK$ is {\em on top under hyeprarithmetic reducibility} if every $\Si^1_1$ equivalence relation on $\om$ hyperarithmetically reduces to the isomorphism problem among computable models of $\KK$.
Another unexpected  empirical observation from the results in \cite{FFHKMM} is that every theory which we could prove was on top under hyperarithmetic reducibility was already on top under effective reducibility.
We show here that this should always be the case, at least among nice theories $T$ where relativization should not be an issue.

\begin{theorem}\label{thm: hyp on top}(ZFC+PD)
If an infinitary sentence $T$ is on top under hyperarithmetic reducibility on a cone, then it is already on top under effective reducibility on a cone.
\end{theorem}

We will prove this theorem at the end of the subsection \ref{ss: Sigma equivalence relations}, as a corollary of Lemma \ref{le: hyp to rec}.
The use of Projective Determinacy ($PD$) is not essential here, and is just to be able to state the theorem saying ``on a cone of Turing degrees,'' instead of ``for co-finally  many Turing degrees''--these two phrases  are equivalent using projective Turing determinacy when the property they are applied to is projective.

%%%%%%%%%
\subsection*{The density property}
Further analyzing the proofs of Theorem \ref{thm: becker}, one can see that the no-intermediate-extension property implies a property that we call the {\em effective-density property}, and is apparently stronger than Vaught's conjecture.
Therefore, if there was going to be a reversal of Theorem \ref{thm: becker}, then the best we can hope for is to prove that these two properties are equivalent, which reminds unknown.
We instead get a reversal from a stronger notion.
Let us now define all these concepts.

\begin{definition}
We say that an $\L_{\om_1,\om}$-sentence $T$ is {\em unbounded} if it has countable models of arbitrary high Scott rank below $\om_1$.
\end{definition}

It is known that an $\L_{\om_1,\om}$-sentence $T$ is bounded if and only if the isomorphism problem among all is countable models is Borel (see \cite[Theorem 12.2.4]{GaoBook}).
It can then be shown that this is also equivalent to having the isomorphism problem among the computable models of $T$ be hyperarithmetic relative to every oracle on a cone.
It thus follows that $T$ has the no-intermediate-extension property if and only if among the extensions $\That$ of $T$, being unbounded is equivalent to being on top under effective reducibility relative to every oracle on a cone.

\begin{definition}
We say that $T$ is {\em minimally unbounded} if it is unbounded, but for every $\L_{\om_1,\om}$-sentence $\varphi$, one of $T\wedge \varphi$ or $T\wedge \neg\varphi$ is bounded.
\end{definition}

It is known (see \cite[Theorem 1.5.11]{Ste78}) that if there is a counter-example to Vaught's conjecture, then there is one that is minimally unbounded.
Such a counterexample is used to build a theory intermediate on a cone in Theorem \ref{thm: becker}.
Let us remark that, as far as we know, minimally unbounded theories do not necessarily have $\aleph_1$ many models, 
and it is unknown whether the existence of a minimally unbounded theory implies the existence of one with $\aleph_1$ models.

A theory which has no minimal extensions is called {\em dense}.
It is unknown whether every unbounded theory is dense.

\subsection*{The effective analogs}
We will need  effective versions of these notions.
Recall that $\om_1^X$ is the least ordinal without an $X$-computable presentation.
When we have an $L_{\om_1,\om}$-formula $\varphi$, we assume we are given a presentation for it, say by a tree describing the structure of the formula.
We can then write $\om_1^\varphi$ for the least ordinal not computable in the real representing $\varphi$.
Or equivalently, $\om_1^\varphi=\min\{\om_1^X: \varphi$ is an $X$-computably infinitary formula $\}$.
For a countable structure $\A$, we let $\om_1^\A=\min\{\om_1^X: X$ computes a copy of $\A\}$, and let $SR(\A)$ be the Scott rank of $\A$ (see subsection \ref{sse: SR} below).

\begin{definition}
We say that an $L_{\om_1,\om}$-sentence $T$ is {\em effectively unbounded} if it has countable models of arbitrary hight Scott rank below $\om_1^{T}$ (i.e.\ for each $\a<\om_1^T$, $T$ has a model of Scott rank at least $\a$).
We say that a structure $\A$ has {\em high Scott rank} if $\om_1^{\A}\leq SR(\A)$. 
\end{definition}

One can show that every satisfiable infinitary sentence $T$ has a countable model $\A$ with $\om_1^\A=\om_1^T$; this follows from Gandy's basis theorem and the fact that  being a model of $T$ is a $\Si^1_1(T)$ property.
We will show in Lemma  \ref{thm: effectively unbounded}  that $T$ is effectively unbounded if and only if it has such a model $\A$ of high Scott rank, that is, satisfying $\om_1^{T}=\om_1^{\A}\leq SR(\A)$.

It is unknown whether being effectively unbounded is different from being unbounded.
This is quite an interesting question.
(See \cite{Sac07} for partial results.)

\begin{definition}
We say that $T$ is {\em effectively minimally unbounded} if it is effectively unbounded, and for every $\L_{\om_1,\om}$-sentence $\varphi$ of quantifier rank less than $\om_1^T$, one of $T\wedge \varphi$ or $T\wedge\neg\varphi$ is bounded below $\om_1^{T}$.
\end{definition}

This is the property that is needed to build a theory that is intermediate for effective reducibility relative to an oracle in Theorem \ref{thm: becker}.
We will show in Theorem \ref{thm: unique model of HSR} that $T$ is effectively minimally unbounded if and only if every oracle $X$ with $\om_1^X=\om_1^T$ computes at most one model of high Scott rank (relative to $X$), and some such $X$ computes at least one.
Considering theories with this property is not really new.
Some time ago, Goncharov and Knight asked whether there existed computably infinitary sentences that have a unique computable model of high Scott rank.
For all the theories researchers have looked at, they have either none, or infinitely many computable model of high Scott rank.

It is unknown whether being effectively minimally unbounded is different from being minimally unbounded

\begin{definition}
We say that  $T$ is {\em effectively dense} if it is unbounded and no extension $\That$ of $T$ is effectively minimally unbounded.
We say that $T$ is {\em uniformly effectively dense} if it is unbounded, for every $\a\in\om_1$ there is a $\b\in\om_1$ such that, for every extension $\That\in\Pii_\a$ of $T$ which is effectively unbounded, there is a $\psi\in\Pii_\b$ witnessing that $\That$ is not effectively minimally unbounded, i.e., such  that both $\That\wedge \psi$ and $\That\wedge \neg\psi$ are unbounded below $\om_1^{\That}$.

Here $\Pii_\a$ refers to the set of infinitary $\Pi_\a$ formulas.
\end{definition}

It is unknown whether being uniformly effectively dense, being effectively dense and being dense are actually different.

We are now ready to state our main theorem.

\begin{theorem}(ZFC+PD) \label{thm:main} 
Let $T$ be an $\L_{\om_1,\om}$-sentence which is uniformly effectively dense.
Then $T$ is on top under effective reducibility, relative to every oracle on a cone. 
\end{theorem}

So we get that having the no-intermediate-extension property is implied by being uniformly effectively dense, and implies being effectively dense.
If it is equivalent to one of these, or if it is strictly in between is unknown.

Projective determinacy (PD) is used in the proofs of the theorems above a few times in the form of Turing determinacy (due to Martin):
If a projective degree-invariant set $\S\subseteq 2^\om$ is {\em co-final in the Turing degrees} (i.e.\ $\forall Z\exists X\geqt Z\ ( X\in\S)$), then $\S$ contains a cone of Turing degrees (i.e.\ $\exists C\forall X\geqt C\ (X\in \S)$).
We did not calculate the exact amount of Turing determinacy is needed in the proofs, nor did we made an effort to optimize it, although, surely much less than the full power of $PD$ is necessary.
In theorems like \ref{thm: linear no int}, it might not be necessary at all.

%%%%%%%%%%%%%%%%%%%%%%%%%%%%%%%
\subsection{Background}
For background on infinitary formulas and computably infinitary formulas, see \cite[Chapter 6 and 7]{AK00}.
We will use $\Sii_\a$ to denote the set of infinitary $\Si_\a$-formulas, $\Sic_\a$ for the computable infinitary formulas, and $\Sic{X}_\a$ for the $X$-computable infinitary formulas.

\subsubsection{Back-and-forth relations}
For more background on the back-and-forth relation see \cite[Chapter 15]{AK00}.
Given structures $\A$ and $\B$, tuples $\abar\in\A^{<\om}$, $\bbar\in\B^{<\om}$ and an ordinal $\xi$, we say that $(\A,\abar)$ is {\em $\xi$-back-and-forth below} $(\B,\bbar)$, and write $(\A,\abar)\leq_\xi(\B,\bbar)$ if the $\Pii_\xi$-type of $\abar$ in $\A$ is contained in the $\Pii_\xi$-type of $\bbar$ of $\B$.
(We are allowing tuples of different sizes here as in \cite{AK00}, provided $|\abar|\leq |\bbar|$.
We note that $(\A,\abar)\leq_\xi(\B,\bbar) \iff (\A,\abar)\leq_\xi(\B,\bbar\upto |\abar|)$.)
Equivalently, $(\A,\abar)\leq_\xi(\B,\bbar) $ if for every tuple $\dbar\in\B^{<\om}$ and any $\g<\xi$, there exists $\cbar\in\A^{<\om}$ such that $(\A,\abar\cbar)\geq_\g(\B,\bbar\dbar)$.

We now review the notion of $\a$-friendliness (see \cite[Section 15.2]{AK00}), which is an ``effectiveness condition'' on a class of structure.
A computable sequence $\{\B_n:n\in\om\}$ of structures is {\em $\a$-friendly} if given two tuples in two structures $\abar\in\B_n^{<\om}$ and $\bbar\in\B_m^{<\om}$ and given $\xi<\a$, we can effectively decide if $(\B_n,\abar)\leq_\xi(\B_m,\bbar)$ in a c.e.\ way, or, in other words, if the set of quintuples $\{(n,\abar,m,\bbar,\xi): n,n\in \om,\abar\in\B_n^{<\om}, \bbar\in\B_m^{<\om}, \xi<\a, \mbox{ such that } (\B_n,\abar)\leq_\xi(\B_m,\bbar)\}$ is c.e.

\subsubsection{Scott rank}\label{sse: SR}
The Scott rank of a structure $\A$ is a measure of its complexity defined as follows.
For each $\abar\in\A^{<\om}$, let $r_\A(\abar)$ be the least $\a$ such that  whenever $\abar\leq_\a\bbar$ for some $\bbar\in \A^{|\abar|}$, we have that $\abar$ and $\bbar$ are automorphic.
We then let $SR(\A)$, the Scott rank of $\A$, to be the least $\a$ greater than $r_\A(\abar)$ for all tuples $\abar\in\A^{<\om}$.
($SR(\A)$ is denoted by $R(\A)$ in \cite[Section 6.7]{AK00}.)
For every structure $\A$, we have that $SR(\A)\leq\om_1^\A+1$ (Nadel \cite{Nad74}), where $\om_1^\A=\min\{\om_1^X: X$ computes a copy of $\A\}$.
Structures with $\om_1^\A\leq SR(\A)$ are said to have {\em high Scott rank}.

When a structure $\A$ has Scott rank $\a$, we have that each automorphism orbit can be defined by a $\Pii_{<\a}$ formula (see \cite[Proposition 6.9]{AK00}).
The collection of these formulas for the different tuples $\abar$ from $\A$ form what is called a Scott family for $\A$.
Given such formulas, one can then define a {\em Scott sentence} for $\A$, which is a sentence that is true about $\A$ and of no other countable structure.
Such formula can be taken to be $\Pii_{\a+1}$.
Conversely, if a structure $\A$ has a $\Pii_{\a+1}$ Scott sentence, then it must have Scott rank $\leq \a+1$. 
The computable structures of high Scott rank are exactly the ones which do not have computably infinitary Scott sentences.
However, it is still true (a due to Nadel \cite{Nad74}, see also \cite[Theorem 7.3]{Bar75}) that if two computable structures satisfy the same computably infinitary sentences,  then they are isomorphic.

\subsubsection{The Harrison linear ordering}
The {\em Harrison linear ordering} is a computable linear ordering, denoted by $\H$, isomorphic to $\om_1^{CK}+\om_1^{CK}\cdot\QQ$ which has no hyperarithmetic descending sequences \cite{Har68}.
The well-founded initial segment, which, abusing notation we denote by $\om_1^{CK}$,  cannot be $\Si^1_1$.
This allows us to use the following kind of argument, called an {\em overspill argument}:
If $P\subseteq \H$ is $\Si^1_1$ and contains the whole initial segment $\om_1^{CK}$, then it also contains some $\a\in\H\sminus\om_1^{CK}$.
We call  such elements $\a$ {\em non-standard ordinals}.

Since the back-and-forth relations are arithmetically definable from the previous ones, one can always define them for $\a\in\H$ beyond $\om_1^{CK}$.
More precisely, fix a computable structure $\A$, and let $P$ be the set of all $\a\in\H$ such that there exists a sequence $\{R_\b:\b\leq\a\}$ of relations $R_\b\subseteq\A^{<\om}\times\A^{<\om}$ which satisfy the definition of the back-and-forth relations, that is,  for all $\b<\a$, $(\abar,\bbar)\in R_\b\iff \forall\d<\b\forall \dbar\exists \cbar\ ((\bbar\dbar,\abar\cbar)\in R_\d)$.
This set $P\subseteq\H$ is $\Si^1_1$ and contains all of $\om_1^{CK}$, and hence contains also some $\a\in \H\sminus\om_1^{CK}$.
The same way, it also makes sense to talk about the notion of $\a$-friendly sequence of structures for $\a\in \H\sminus\om_1^{CK}$.

We remark that all these notions can be relativized.
We use $\H^X$ to denote the Harrison linear ordering relative to $X$.

%%%%%%%%%%%%%%%%%%%%%%%%%%%%%%%%%%%%%%%%%%%%%%%%%%%%%%%%%%%%%%%%%%%%%%%%%%%%%%%%%%%%%%%%%%%%%%%%%%%%%%%%%%%%%%%%%%%%%%%%%%%%%%%%%%%%%%%%%%%%%%%%%%%%%%%%%%%%%%%%%%%%%%%%%%%%%%%%%%%%%%%%%%%%%%%%%%%%%%%%%%%%%%%%%%%%%%%%%%%%%%%%%%%%%%%%%%%%%%%%%%%%%%%%%%%%%%%%%%%%%%%%%%%%%%%%%%%%%%%%%%%%%%%%%%%%%%%%%%%%%%%%%%%%%%%%%%%%%%%%%%%%%%%%%%
%\section{The known implications}

%%%%%%%%%%%%%%%%%%%%%%%%%%%%%%%%
\section{Models of high Scott rank}

In this section, we quickly prove the results about structures of high Scott rank mentioned in the introduction.
In particular the characterization of effectively unbounded theories, and  effectively minimally unbounded theories, in terms of models of high Scott rank.

\begin{lemma}\label{thm: effectively unbounded}
An infinitary sentence $T$ is effectively unbounded if and only if it has a model $\A$ with $\om_1^T=\om_1^\A\leq SR(\A)$.
\end{lemma}
\begin{proof}
The right-to-left direction is immediate from the definition of effectively unbounded.
For the left-to-right consider, for each $\a<\om_1^T$, the sentence $S_\a$ such that $\A\models S_\a$ if and only if $SR(\A)\geq\a$.
It is well-known such sentences exist and can be taken to be $T$-computably infinitary.
Since $T$ is effectively unbounded, for every $\a<\om_1^T$, $T\cup\{S_\a\}$ has a model.
By Barwise compactness $T\cup\{S_\a:\a<\om_1^T\}$ has a model.
Any such model $\A$ would satisfy $\om_1^T\leq SR(\A)$.
Since being a model of $T\cup\{S_\a:\a<\om_1^T\}$ is a $\Si^1_1(T)$ property, by Gandy's basis theorem, there is such a model $\A$ with $\om_1^T=\om_1^\A$.
\end{proof}

Furthermore, we can assume that also $\om_1^T=\om_1^{T,\A}$, where $\om_1^{T,\A}=\min\{\om_1^X: X $ computes a presentation of $\A$ and $T$ is an $X$-computably infinitary sentence$\}$.

The Scott sentence of a structure is the one that identifies a structure up to isomorphism, among all countable structures.
If all we want is to identify a structure up to its $\a$-back-and-forth type, a simpler sentence can be used.
Suppose that $\A$ has a computable copy.
Recall that $\A\leq_\a\B$ if and only if the $\Pii_\a$-theory of $\A$ is a subset of the one of $\B$.
However, the assumption that the $\Pic_\a$-theory of $\A$ is a subset of the one of $\B$ is not enough to obtain $\A\leq_\a\B$.
The following lemma gives us a good  approximation.

\begin{lemma}\label{le: computable bf}
Let $\A$ be a computable structure, and $\B$ be any structure.
\begin{enumerate}
\item If $\Sic_{3\cdot\a}$-Th$(\A)\subseteq \Sic_{3\cdot\a}$-Th$(\B)$,  then  $\A\geq_\a\B$.
\item If $\Pic_{3\cdot\a}$-Th$(\A)\subseteq \Pic_{3\cdot\a}$-Th$(\B)$, then  $\A\leq_\a\B$.
\end{enumerate}
\end{lemma}
\begin{proof}
The proof is by transfinite induction.
Suppose first that $\Sic_{3\cdot\a}$-Th$(\A)\subseteq \Sic_{3\cdot\a}$-Th$(\B)$, and we want to show that $\A\geq_\a\B$. 
Take $\abar\in\A^{<\om}$ and $\d<\a$.
Let  $\psi_{\abar}(\xbar)$ be the conjunction of all the $\Pic_{3\d}$-formulas true about $\abar$ in $\A$.
This set of formulas is $\Pi^0_{3\d}$, and hence this conjunction is equivalent to a $\Pic_{3\d+1}$ formula (see \cite[Proposition 7.12]{AK00}), and the formula $\exists \xbar\psi_{\abar}(\xbar)$ is, in particular, $\Sic_{3\a}$.
Since it is true in $\A$ it is true in $\B$, and hence there is $\bbar$ in $\B$ such that $\B\models\psi_{\abar}(\bbar)$.
But then $\Pic_{3\d}$-tp$_\B(\bbar)\supseteq\Pic_{3\d}$-tp$_\A(\abar)$, and hence by the inductive hypothesis that $(\A,\abar)\leq_\d(\B,\bbar)$.

Suppose that $\Pic_{3\cdot\a}$-Th$(\A)\subseteq \Pic_{3\cdot\a}$-Th$(\B)$, and we want to show that $\A\leq_\a\B$. 
Take $\bbar\in\B^{<\om}$ and $\d<\a$.
For each $\abar\in\A^{<\om}$ let $\psi_{\abar}$ be now the $\Pic_{3\d+1}$ formula equivalent to the conjunction of all the $\Sic_{3\d}$-formulas true about $\abar$.
Then $\A$ models $\forall\xbar\bigvee_{\abar\in\A}\psi_{\abar}(\xbar)$.
This is a $\Pic_{3\d+3}$ sentence, and hence it is true about $\B$, too.
So, there is some $\abar$ such that $\B\models \psi_{\abar}(\bbar)$, and hence $\Sic_{3\cdot\a}$-Th$(\A,\abar)\subseteq \Sic_{3\cdot\a}$-Th$(\B,\bbar)$.
By the inductive hypothesis we then get that $(\A,\abar)\geq_\d(\B,\bbar)$.
\end{proof}

Note that if $\a$ is a limit ordinal, then $3\a=\a$.

We are now ready to prove the characterization of effectively minimally unbounded theories.

\begin{theorem}\label{thm: unique model of HSR}
An infinitary sentence $T$ is effectively minimally unbounded if and only if every oracle $X$ with $\om_1^X=\om_1^T$ computes at most one model of high Scott rank (relative to $X$), and some such $X$ computes at least one.
\end{theorem}
\begin{proof}
Suppose first that $T$ is effectively minimally unbounded.
Since it is effectively unbounded, there is at least one model $\A$ of $T$ with $\om_1^T=\om_1^{T,\A}\leq SR(\A)$ and some $X$ with $\om_1^{T,X}=\om_1^T$ which computes a presentation for it.
Suppose $\B$ was another such model computable from $X$.
Then, $\A$ and $\B$ satisfy the same $X$-computably infinitary sentence:
This is because, for every $X$-computably infinitary sentence $\varphi$, one of $T\wedge\varphi$ and $T\wedge\neg\varphi$ is bounded below $\om_1^T$, and hence the other one is true in both $\A$ and $\B$.
It follows that $\A$ and $\B$ are isomorphic.

Suppose now that $T$ is not effectively minimally unbounded.
If $T$ is not even effectively unbounded, then, by the Lemma \ref{thm: effectively unbounded}, no $X$ with $\om_1^X=\om_1^T$ computes a model of $T$ of high Scott rank.
Suppose then that it is effectively unbounded and that there is a $\Pii_\a$-sentence $\varphi$ with $\a<\om_1^T$ such that  $T\wedge\varphi$ and $T\wedge\neg\varphi$ are both unbounded below $\om_1^T$.
If we had that $\om_1^{T\wedge\varphi}=\om_1^{T}$, then we easily could directly (applying Barwise compactness and Gandy's basis theorem) find and $X$ with $\om_1^X=\om_1^T$ which computes a two model of $T$ of high Scott rank, one satisfying $\varphi$ and one satisfying $\neg\varphi$.
However, there is no reason to assume that $\om_1^{T\wedge\varphi}=\om_1^{T}$.
We will show that we can use find another formula $\psi$ that also splits $T$ in two effectively unbounded theories, but with $\om_1^{T\wedge\psi}=\om_1^{T}$.

Let $X$ be an oracle with $\om_1^{X,T}=\om_1^T$ and which computes a model $\A$ of high Scott rank, i.e., $\om_1^T=\om_1^\A\leq SR(\A)$.
Then either $\varphi$ or $\neg \varphi$ is true in $\A$; suppose it is $\varphi$.
Let $\psi$ be the conjunction of the whole $\Pico{X}_{3\a}$-theory of $\A$.
For any model $\B$ of $\neg\varphi$ we have $\B\not\geq_\a\A$, and hence $\B\models\neg\psi$ by Lemma \ref{le: computable bf}.
Thus, since $T\wedge\neg\varphi$ is unbounded below $\om_1^T$, so is $T\wedge\neg\psi$.
Since $\om_1^{X,T\wedge\psi}=\om_1^T$ (because $\psi$ is hyperarithmetic in $X$), there is a model $\B\models T\wedge\neg\varphi$ such that $\om_1^T=\om_1^{X,\B}\leq SR(\B)$.
Let $Y\geqt X$ compute a copy of $\B$ and satisfy $\om_1^Y=\om_1^T$.
This $Y$ contradicts the right-hand-side of the theorem as it computes two different models of high Scott rank.
\end{proof}

%%%%%%%%%%%%%%%%%%%%%%%%%%%%%%%%%%%%%%%%%%%%%%%%%%%%%%%%%%%%%%%%%%%%%%%%%%%%%%%%%%%%%%%%%%%%%%%%%%%%%%%%%%%%%%%%%%%%%%%%%%%%%%%%%%%%%%%%%%%%%%%%%%%%%%%%%%%%%%%%%%%%%%%%%%%%%%%%%%%%%%%%%%%%%%%%%%%%%%%%%%%%%%%%%%
\section{The proof of the main theorem}

In this section, we prove Theorem \ref{thm:main}.
That is, assuming $T$ is uniformly effectively dense, we will show that there is a cone such that, relative to every oracle on that cone, $T$ is on top for effective reducibility.
This proof is divided in several steps.
First, in Subsection \ref{ss: Sigma equivalence relations} we study a particular way of representing $\Si^1_1$ equivalence relations on $\om$ using transfinite binary sequences.
In Subsection \ref{ss: tree of structures} we consider trees of structures, where the structures are indexed by transfinite binary sequences, and we show how to use them to define reductions from $\Si^1_1$-equivalence relations to structures.
In Subsection \ref{ss: functions on ordinals}, we deal with a different aspect of the proof which has to do with finding computable representation for functions from ordinals to ordinals. 
In Subsection \ref{ss: building trees}, we go back to the trees of structures, and we show how to build them when we have a uniformly effectively dense theory.
We finally put all the ingredients together in Subsection \ref{ss: tying}.

%%%%%%%%%%%%%%%%%%%%%%%%%%%%%%%%
\subsection{A representation of $\Si^1_1$-equivalence structures} \label{ss: Sigma equivalence relations}

To prove that $T$ is on top under effective reduction we need to define an embedding from an arbitrary $\Si^1_1$ equivalence relation on $\om$ into the computable models of $T$.
We start by finding a particular representation of an arbitrary $\Si^1_1$ equivalence relation that will be useful to build this embedding.

The first lemma allows us to approximate $\Si^1_1$ equivalence relations by hyperarithmetic ones.

\begin{lemma}\label{le: presentation of simga eq}
For every $\Si^1_1$-equivalence relation $\sim$ of $\om$, there is a sequence $\{\sim_\a:\a<\om_1^{CK}\}$ of equivalence relations such that, for all $n,m\in\om$,
\begin{itemize}
\item $n\sim m\iff (\forall \a\in\om_1^{CK})\ n\sim_\a m$.
\item For $\b\leq\a$, $n\sim_\a m \implies n\sim_\b m$.
\item Each $\sim_\b$ is $\Si^0_{\b+1}$ uniformly in $\b$.
\end{itemize}
\end{lemma}
\begin{proof}
The Borel version of this result of for analytic equivalence relations on the reals is due to Burgess \cite[Corollary 1]{Bur79}, but he only required each $\sim_\b$ to be Borel and not necessarily $\Si^0_{\b+1}$ uniformly in $\b$.
A proof of this exact lemma, but for equivalence relations on reals, can be found in \cite{MonAnalyticER}.
If one codes each natural number by a real (say the one that has a 1 at position $n$, and 0's elsewhere), then the lemma from \cite{MonAnalyticER} applies here too.
\end{proof}

\begin{definition}
Let $\fseq{\a}$ to be the set of all $\a$-long binary sequences $\si\in 2^\a$ with only finitely many 1's.
\end{definition}

Notice that $\fseq{\alpha}$ is countable and computably presentable whenever $\a$ is itself computable, as opposed to $2^\alpha$ which has size continuum for infinite $\alpha$.

\begin{definition}
For a computable ordinal $\a$, we say that a sequence $\{\si_n:n\in\om\}\subseteq 2^{\circ\a}$ is {\em uniformly $\Si^0_{\xi\mapsto\xi+1}$} if deciding if $\si(\xi)_n=1$ is $\Si^0_{\xi+1}$ uniformly in $\xi$ and $n$,
or, in other words, if there is a c.e.\ operator $W$, such that $\si_n(\xi)=1\iff n\in W^{\nabla^{\xi}}$, where $\nabla^\xi$ is a complete $\Delta^0_\xi$ real (see \cite{MonAsh}). 
\end{definition}

The definition above does not require $\a$ to be an ordinal, but just that the iterations of the jump, $\nabla^\xi$, exist for $\xi<\a$.
So, if we assume that $\nabla^\xi$ exists for each $\xi$ in the Harrison linear ordering, $\H$, then we can still talk about uniformly $\Si^0_{\xi\mapsto\xi+1}$ sequences in $\fseq{\H}$.

\begin{lemma}\label{le: sequences on top}
For each $\Si_1^1$-equivalence relation $\sim$ on $\om$, there exists a uniformly $\Si^0_{\xi\mapsto\xi+1}$ sequence $\{\si_n:n\in\om\}\subseteq \fseq{\H}$,  
such that
\[
(\forall n,m\in\om) \ n\sim m\iff \si_n\upto\om_1^{CK}=\si_m\upto\om_1^{CK}.
\]
\end{lemma}
\begin{proof}
We will define $\si_n(\xi)$ by transfinite recursion on $\xi$.
The general idea is as follows.
Suppose we have already defined $\si_n\upto \xi$ for all $n$.
So, we have an equivalence relation $E_\xi$ on $\om$ given by $n\  E_\xi\ m$ if $\si_n\upto\xi=\si_m\upto \xi$.
At stage $\xi$ we preserve the inclusion $\sim_\xi\subseteq E_\xi$, and we only take one step towards making $E_\xi$ closer to $\sim_\xi$ as follows.
Each $E_\xi$-equivalence class consists of many (possibly just one) $\sim_\xi$-equivalence classes.
If it is only one, we are in good shape and we do not do anything.
Within each $E_\xi$-equivalence class which contains at least two $\sim_\xi$ equivalence classes, we will define $\si_n(\xi)$ to be $0$ or $1$ so that we split the $E_\xi$-equivalence class into two $E_{\xi+1}$-equivalence classes, by separating the first $\sim_\xi$-equivalence class from the rest.
We will actually consider at $\sim_{\xi-1}$ instead of $\sim_\xi$ to keep the complexity low.

More concretely:
\begin{quote}
Let $\si_n(\xi)=1$ if for the least  $m< n$ with $\si_m\upto\xi=\si_n\upto\xi$, there is some $\b<\xi$ such that  $n\not\sim_\b m$, and let $\si_n(\xi)=0$ otherwise.
\end{quote}
By counting quantifiers, it is not hard to see that $\si_n$ is uniformly $\Si^0_{\xi\mapsto\xi+1}$.

Take $n_0,n_1\in\om$, and suppose that $\si_{n_0}\upto\om_1^{CK}\not=\si_{n_1}\upto\om_1^{CK}$.
Let $\xi$ be the first value where $\si_{n_0}(\xi)\neq\si_{n_1}(\xi)$.
Suppose $\si_{n_0}(\xi)=0$ and $\si_{n_1}(\xi)=1$.
Let $m$ be the least  with $\si_m\upto\xi=\si_{n_0}\upto\xi=\si_{n_1}\upto\xi$.
From the definition of $\si_{n_0}(\xi)$ and $\si_{n_1}(\xi)$, we get that for some $\b<\xi$, $n_1\not\sim_\b m\sim_\b n_0$, and hence $n_0\not\sim n_1$.

Suppose now that $m<n$, $\si_m\upto\om_1^{CK}=\si_n\upto\om_1^{CK}$, and, towards a contradiction, that $m\not\sim n$.
Suppose that $m$ is the least for which there exists such an $n$.
Thus, if there was some $n_0<m$ with $\si_{n_0}\upto\om_1^{CK}=\si_m\upto\om_1^{CK}$, we would have  $n_0\sim m$ and $n_0\sim n$.
So we can assume that $m$ is the least such that  $\si_{m}\upto\om_1^{CK}=\si_n\upto\om_1^{CK}$.
For some $\b<\om_1^{CK}$ high enough, we have that $n\not\sim_\b m$, and $m$ is still the least with $\si_{m}\upto\b=\si_n\upto\b$.
Let $\xi=\b+1$.
Then, by definition of $\si_n(\xi)$, we would get $\si_n(\xi)=1$ and $\si_m(\xi)=0$ contradicting that $\si_m\upto\om_1^{CK}=\si_n\upto\om_1^{CK}$.
\end{proof}

Knight and Montalb\'an \cite{KMonTop} have already shown that $\Si^1_1$ equivalence relations could be represented by uniformly $\Si^0_{\xi\mapsto\xi+1}$ sequences in $2^\H$, possibly with infinitely many 1's.
The fact that using sequences with finitely many 1's is enough, is important for the rest of the paper.

Aa a corollary of this lemma, we can now prove Theorem \ref{thm: hyp on top}.
We first prove the following result that avoids the use of Turing determinacy, and implies Theorem \ref{thm: hyp on top} using Turing determinacy.

\begin{lemma}\label{le: hyp to rec}
If the set of oracles, relative to which $T$ is on top under hyperarithmetic reducibility, is co-final in the Turing degrees, 
then so is the set of oracles relative to which $T$ is on top under effective reducibility.
\end{lemma}
\begin{proof}
Let $C$ be any real.
We want to show that there is some $X\geqt C$ relative to which $\KK$ is on top under effective reducibility.
By hypothesis we might assume $C$ is such that $T$ is on top under hyperarithmetic reducibility relative to $C$.

Consider the following equivalence relation on $\om$.
First take a non-standard ordinal $\a^*\in \H^C\sminus\om_1^C$.
For each $e\in\om$, let $\si_e$ be the sequence in $\fseq{\a^*}$ for which $e$ is a $\Si^C_{\xi\mapsto\xi+1}$-code.
In other words, let $\si_e(\xi)=1\iff \xi\in W_{e}^{\nabla^{\xi}}$, where $\nabla^\xi(C)$ is a complete $\Delta^{0}_\xi(C)$ real and $W_e$ is the $e$th c.e.\ operator.
Given $e_0,e_1\in\om$, let $e_0\sim^C e_1$ if $\si_{e_0}\upto\om_1^{CK}=\si_{e_1}\upto\om_1^{CK}$.
Notice that this is a $\Si^1_1$-equivalence relation.
This is the equivalence relation Knight and Montalb\'an had considered in \cite{KMonTop}, and proved that it is on top under effective reducibility (relative to $C$), which now follows from Lemma \ref{le: sequences on top}.
Just because it is $\Si^1_1$, there is a $C$-hyperarithmetic reduction $h$ from $\om$ to $C$-computable indices of structures in $\KK$, such that $e_0\sim^C e_1\iff \A^C_{h(e_0)}\isom \A^C_{h(e_1)}$ (where $\A^C_n$ is the structure coded by the $n$-th Turing machine with oracle $C$).
For some $\b<\om_1^C$, $h$ is $\Delta^\b_0(C)$.
Let $X$ be $\nabla^\b(C)$.
We will define an $X$-computable function $f$ such that $i_0\sim^X i_1\iff \A^X_{f(i_0)}\isom \A^X_{f(i_1)}$, which would then imply that $\KK$ is on top under effective reducibility relative to $X$.
Let $i$ be a $\Si^X_{\xi\mapsto\xi+1}$-code for a sequence  $\si\in \fseq{\a^*}$.
Let $\sihat$ consist of a string of $\beta$ many 0's followed by $\si$ (that is $\sihat(\g)=0$ if $\g<\b$ and $\sihat(\b+\g)=\si(\g)$).
Find an index $e$ for $\sihat$ as a $\Si^C_{\xi\mapsto\xi+1}$ sequence and let $g(i)=e$.
We have defined a function $g\colon\om\to\om$ such that $i_0\sim^X i_1 \iff g(i_0)\sim^C g(i_1)$.
 Let $f(i)$ be an $X$-index for $\A^C_{h(g(i))}$ viewed as an $X$-computable structure.
Thus, we have that $\A^X_{f(i)}= \A^C_{h(g(i))}$.
We then have that 
\[
i_0\sim^X i_1 \iff g(i_0)\sim^C g(i_1) \iff \A^C_{h(g(i_0))}\isom \A^C_{h(g(i_1))} \iff \A^X_{f(i_0)}\isom \A^X_{f(i_1)},
\]
as wanted.
\end{proof}

\begin{proof}[Proof of Theorem \ref{thm: hyp on top}]
Apply projective Turing determinacy to the set of oracles $X$ relative to which $T$ on top under effective reducibility.
\end{proof}

%%%%%%%%%%%%%%%%%%%%%%%%%%%%%%%%%
\subsection{Trees of structures} \label{ss: tree of structures}

Now that we can represent $\Si^1_1$-equivalence relations in terms of uniformly $\Si^0_{\xi\mapsto\xi+1}$ sequences in $\fseq{\H}$, we need to associate these sequences with structures.

\begin{definition}\label{def: eta tree}
For an ordinal $\eta$, an {\em $\eta$-tree of structures} is a sequence of structures $\{\A_\si: \si\in \fseq{\eta}\}$ such that, for every $\si,\tau\in \fseq{\eta}$ and $\xi<\eta$, we have that 
\[
\si\upto \xi=\tau\upto \xi 
\quad \implies \quad
\A_{\si}\equiv_{\xi+1}\A_{\tau}.
\]
\end{definition}

We will show that when a class of structures has arbitrary long non-trivial trees of structures, the class is on top under effective reducibility.
Of course, to get non-trivial trees of structures we have to ask that all structures $\A_\si$ are non-isomorphic. 
But we have to be careful with this, as we do not want to use $\Pi^1_1$-properties in the definition.

The following is a generalization of Ash--Knight's theorem on pairs of structures to trees of structures.
It says that if we have an $\eta$-friendly $\eta$-tree of structures, and we are given an index for a $\Si^0_{\xi\mapsto\xi+1}$ sequence $\si\in\fseq{\eta}$, we can uniformly computably build a copy of $\A_\si$.
Thus, even if guessing the bits of $\si$ is complicated, namely $\Si^0_{\xi\mapsto\xi+1}$, then we can still produce a computable copy of $\A_\si$.

\begin{theorem}[\cite{MonAsh}]\label{thm: app}
Let $\{\A_\si:\si\in\fseq{\eta}\}$ be a computable $\eta$-friendly $\eta$-tree of structures.
Let $\{\si_n:n\in\om\}\subseteq \fseq{\eta}$ be uniformly $\Si^0_{\xi\mapsto\xi+1}$.
Then, there exists a computable sequence of computable structures $\{\C_n:n\in\om\}$ such that for all $n$, $\C_n\isom\A_{\si_n}$.
\end{theorem}
\begin{proof}
The result in \cite{MonAsh} is slightly finer than this.
In there, it is assumed that $\si\upto \xi=\tau\upto \xi \and \si(\xi)\leq\tau(\xi) \implies \A_{\si}\geq_{\xi+1}\A_{\tau}$, the conclusion being the same.
That assumption still holds with our definition of $\eta$-tree.
\end{proof}

Again in Definition \ref{def: eta tree}, the fact that $\eta$ is an ordinal is not essential so long as we can talk about the $\xi$-back-and-forth relations for every $\xi<\eta$.
On any computable family of structures, one can always define these relations on an initial segment of $\H$ which is longer than $\om_1^{CK}$.
Let us notice that if  $\a^*$ is a computable pseudo-well-ordering, $\a^*\in\H\sminus\om_1^{CK}$, and we have a computable $\a^*$-tree of structures $\{\A_\si: \si\in \fseq{\a^*}\}$, then whenever $\si\upto\om_1^{CK}=\tau\upto\om_1^{CK}$,  $\A_\si\isom \A_\tau$.
This is because we would have that $\A_\si\equiv_{\om_1^{CK}}\A_\tau$, which implies they are isomorphic.

\begin{definition}
For $\a^*\in\H\sminus\om_1^{CK}$, we say that an $\a^*$-tree of structures $\{\A_\si: \si\in \fseq{\a^*}\}$ is {\em proper} if for ever $\si,\tau\in\fseq{\a^*}$,
\[
\si\upto\om_1^{CK}=\tau\upto\om_1^{CK} \iff \A_\si\isom \A_\tau.
\]
\end{definition}

The following theorem shows how trees of structures are used to get reductions from $\Si^1_1$-equivalence relations.

\begin{theorem} \label{thm: tree to top}
Suppose that there exists a computable, proper, $\a^*$-friendly $\a^*$-tree of models of $T$ for some $\a^*\in\H\sminus\om_1^{CK}$.
Then $T$ is on top under effective reducibility.
\end{theorem}
\begin{proof}
Let $\sim$ be a $\Si^1_1$ equivalence relation on $\om$.
We need to build a sequence $\{\C_n:n\in\om\}$ of computable models of $T$ such that $n \sim m\iff \C_n\isom\C_m$.

Let $\{\si_n:n\in\om\}\subseteq\fseq{\a^*}$ be a uniformly $\Si^0_{\xi\mapsto\xi+1}$ sequence such that $(\forall n,m\in\om) \ n\sim m\iff \si_n\upto\om_1^{CK}=\si_m\upto\om_1^{CK}$ as given by Lemma \ref{le: sequences on top}.
We will now apply an overspill argument to the theorem above.
For each  $\b\in\a^*$, and $n\in\om$, let $\si_{n,\b} \in \fseq{\a^*}$ be defined by copying $\si_n$ up to $\b$ and extending to $\a^*$ with $0$'s (i.e.\ $\si_{n,\b}(\g)=\si_n(\g)$ if $\g<\b$ and $\si_{n,\b}(\g)=0$ if $\g\geq\b$).

Let $P$ be the set of all $\b\in \a^*$ such that there exists a computable sequence $\{\C_n:n\in\om\}$ such that $\C_n\isom\A_{\si_{n,\b}}$ for all $n\in\om$ and for all $\xi<\b$.
The set $P$ is $\Si^1_1$.
The set $P$ contains all ordinals $\b<\om_1^{CK}$ by Theorem \ref{thm: app} applied to the $\b$-tree obtained by truncating the $\a^*$ tree.
Thus, there is a non-standard ordinal $\b^*\in P\setminus \om_1^{CK}$ together with a witnessing sequence $\{\C_n:n\in\om\}$ satisfying that $\C_n\isom\A_{\si_{n,\b^*}}$ for every $n$.
Now, for each $n$, $\A_{\si_n}\isom\A_{\si_{n,\b^*}}$ because $\si_n\upto \om_1^{CK}=\si_{n,\b^*}\upto \om_1^{CK}$.
Thus, $\C_n\isom\A_{\si_n}$ as needed.
\end{proof}

%%%%%%%%%%%%%%%%%%%%%%%%%%%%%%%
\subsection{Functions from ordinals to ordinals} \label{ss: functions on ordinals}
The next objective is  be to build such $\a^*$-trees.
But before that we need a lemma about the representation of functions from ordinals to ordinals.

\begin{definition}
We say that $f\colon\om_1\to\om_1$ {\em witnesses that $T$ is uniformly effectively dense} if for every $\a\in\om_1$ and every $\varphi\in\Pii_\a$ such that $T\wedge\varphi$ is effectively unbounded, there is a $\psi\in\Pii_{f(\a)}$ such  that both $T\wedge\varphi\wedge \psi$ and $T\wedge\varphi\wedge \neg\psi$ are unbounded below $\om_1^{T\wedge\varphi}$.
\end{definition}

Notice that if $T$ is uniformly effectively dense, then there is a projective representation for such an $f$.
By that we mean a projective subset $F\colon WO\times WO$ (where $WO$ is the set of well-orderings of $\om$) such that for every $A\in WO$, $f(|A|)=\b$ if and only if there exists some $B\in WO$ with $|B|=\b$ and $(A,B)\in F$ (where $|A|$ is the ordinal in $\om_1$ of the same order type as $A$).
However,  for our argument we will need $f$ to be much simpler than projective.
Under enough determinacy assumptions, one  can always find a much simpler presentation for $f$.

\begin{definition}
We say that $f\colon\om_1\to\om_1$  {\em looks computable according to $X\in 2^\om$} if $f$ maps ordinals below $\om_1^X$ to ordinals below $\om_1^X$, and on some $X$-computable linear ordering $\a^*$, which has an initial segment isomorphic to $\om_1^{CK}$ (i.e.\ a Harrison linear ordering), $X$ can compute a function $f^X\colon\a^*\to\a^*$ which coincides with $f$ on $\om_1^X$.
\end{definition}

\begin{theorem}(ZF+PD) \label{thm: looks computable}
For every function $f\colon\om_1\to\om_1$ with a projective presentation there is a cone such that $f$ looks computable according to every $X$ on that cone.
\end{theorem}
\begin{proof}
First, we claim that there is an oracle $Y$ such that every $Y$-admissible ordinal is closed under $f$.
This follows from PD and the fact that  the set of ordinals $\a\in\om_1$ such that $\a$ is closed under $f$ forms a club, which is projective when viewed as a subset of $WO$:
Consider the set of all $X$ such that $\om_1^X$ is closed under $f$.
By projective Turing determinacy there is a cone, say with base $Y$, that is either contained in or disjoint form this set. 
Sacks proved that the $Y$-admissible ordinals are exactly the ones of the form $\om_1^X$ for some $X\geqt Y$.
Thus, either every $Y$-admissible ordinal is closed under $f$ or none is.
But, since the $Y$-admissible ordinals contain a club, and so do the ordinals closed under $f$, there is at least one $Y$-admissible ordinal closed under $f$.
But then they all are.
Let us relativize the rest of the proof to such $Y$, and assume that every admissible ordinal is closed under $f$.

Let $\S$ be the set of all $X$ according to which $f$ looks computable.
This set is projective, and by projective Turing determinacy, all we need to do is to show that is co-final in the Turing degrees, i.e., that $\forall Z\exists X\geqt Z\ (X\in \S)$. 
We relativize the rest of the proof to such $Z$, so all we have to do is show that there is some $X\in \S$.

Consider $L_{\om_1}[f]$, where $f$ is viewed as s relation symbol, and $L_{\a+1}[f]$ is defined to be the set of definable subsets of $(L_\a[f];\in,f\cap\a\times\a)$ (see, for instance, \cite[Section 1.3]{Kan03}).
Let $\a$ be such that $L_\a[f]$ is admissible and every ordinal is countable inside $L_\a[f]$.
(For instance let $\a=\om_1^{L[f]}$.)
Now, using Barwise compactness for the admissible set $L_{\a}[f]$ \cite[Theorem III.5.6]{Bar75} we get an ill-founded model $\M=(M;\in^\M,f^\M)$ of $KP$ whose ordinals have well-founded part equal to $\a$, with $f^\M\upto \a$ coinciding with $f\upto\a$, and satisfying that every ordinal can be coded by a real.
(To show this one has to consider the infinitary theory in the language $L=\{\in, f,c\}$ saying all this, plus axioms saying that the constant symbol $c$ is an ordinal and that any ordinal below $\a$ exists and is below $c$.
Then  observe that whole the set of axioms is $\Si_1(L_\a[f])$, and that, choosing by $c$ appropriately, $L_\a[f]$ is a model of any subset of these axioms which is a set in $L_\a[f]$.
Thus, by Barwise compactness \cite[Theorem III.5.6]{Bar75}, this theory has a model and its ordinals have well-founded part at least $\a$.
Then, using \cite[Theorem III.7.5]{Bar75}, we get such a model with well-founded part exactly $\a$.)
Let $\a^*$ be a non-standard ordinal in $\M$, i.e., $\a^*\in ON^\M\sminus \a$, and let $X$ be a real in $\M$ coding $\a^*$ and $f^\M\upto \a^*$.
Notice that $\om_1^X=\a$.
(To see this, we have that $\om_1^X\geq\a$ because it codes every initial segment of $\a$, and $\om_1^X\leq\a$ because every $X$-computable well-ordering is isomorphic to an ordinal in $\M$ and hence below $\a$.)
This shows that  $f$ looks computable according to $X$.
\end{proof}

%%%%%%%%%%%%%%%%%%%%%%%%%%%%
\subsection{Building a tree of structures} \label{ss: building trees}

Suppose $T$ is  uniformly effectively dense witnessed by $f$.
To be able to apply Theorem \ref{thm: tree to top} we would like to build, for each $X$ on a cone, a computable, proper, $\a^*$-friendly $\a^*$-tree of models of $T$ for some non-standard $\a^*\in \H^X\sminus \om_1^X$.
For this we would like to use an overspill argument, but the first problem we encounter is that being ``proper'' is a $\Pi^1_1$ property.
For that reason, we consider the notion of $g$-proper, which is $\Delta^1_1$.

\begin{definition}
Given $g\colon \om_1\to\om_1$ and $\eta\in\om_1$, we say that an $\eta$-tree $\{\A_\si:\si\in\fseq{\eta}\}$ is {\em $g$-proper} if for every $\xi<\eta$, if $\si\upto\xi\neq\tau\upto \xi$, then $\A_{\si}\not\equiv_{g(\xi)}\A_{\tau}$.
\end{definition}

We remark that, on one hand, being a $g$-proper tree is a $\Delta^1_1$ property (relative to $g$).
On the other hand, if we have $\a^*$-tree of models of $T$ for some computable non-standard $\a^*\in \H\sminus \om_1^{CK}$, which satisfies the definition of ``$g$-proper tree'' for $\xi<\om_1^{CK}$, then we know the tree is actually proper.

The function  $g$ we are going to use is defined by iterating $f$.
That is, for $\b\in\om_1$, 
\[
g(\b)=\sup_{\g<\b}f(g(\g)+1) +\om.
\]

Without loss of generality, we will assume that for all $\b$, $\b\leq f(\b)$.
The same is then true for $g$.
We remark that the definition of $g$ is far from being optimal.

Before considering non-standard trees, we want to show that, for every $X$ on a cone and every $\a<\om_1^X$, $X$ computes an $g$-proper, $\a$-friendly $\a$-tree.
The first step is to show that $g$-proper $\a$-trees exists.

\begin{lemma} \label{le: alpha tree}
Assume $T$ is uniformly effectively dense witnessed by $f\colon\om_1\to\om_1$, and let $g$ be defined by iterating $f$ as above.
For every $\a\in\om_1$, there is a $g$-proper $\a$-tree.
\end{lemma}
\begin{proof}
Let $X$ be such that $\a<\om_1^X$, and such that $g$ looks computable according to $X$ (which exists by Theorem \ref{thm: looks computable}).
For each $\si\in\fseq{\a}$ we will define a structure $\A_\si$ such that $\om_1^X=\om_1^{\A_\si}\leq SR(\A_{\si})$.
We define the structures $\A_\si$ by induction on the number of 1's in $\si$.
For $\si$ the $\a$-string of all $0$s, let $\A_\si$ be any structure with $\om_1^X=\om_1^{\A_\si}\leq SR(\A_{\si})$ which we know exists using that $T$ is unbounded and Lemma \ref{thm: effectively unbounded}.

Suppose now that we have $\si\in\fseq{\a}$ and we need to define $\A_\si$.
Let $\xi<\a$ be the largest with $\si(\xi)=1$, and let $\si^-$ be defined be making that `1' into a `0', that is, $\si^-(\gamma)=\si(\gamma)$ if $\g\neq\xi$ and $\si^-(\xi)=0$.
By induction, we can assume that we have already defined $\A_{\si^-}$ of high Scott rank, and that we have a presentation computable in some $Y$ with $\om_1^Y=\om_1^X$.
We will define $\A_\si$ so that $\A_\si\equiv_{g(\xi)}\A_{\si^-}$, and $\A_\si\not\equiv_{g(\xi+1})\A_{\si^-}$.
Let $\theta _0$ be the conjunction of the $\Pico{Y}_{g(\xi)}$ and $\Sico{Y}_{g(\xi)}$ theories of $\A_{\si^-}$, and $\theta_1$ be the conjunction of the $\Pico{Y}_{3f(g(\xi)+1)}$ and $\Sico{Y}_{3f(g(\xi)+1)}$  theories of $\A_{\si^-}$.
Lemma \ref{le: computable bf} then implies that for any $\B\models \theta_0\wedge\neg\theta_1$ we have $\B\equiv_{g(\xi)}\A_{\si^-}$, and $\B\not\equiv_{g(\xi+1})\A_{\si^-}$ (we are using here that $g(\xi)$ is a limit ordinal and hence that $3g(\xi)=g(\xi)$, and we are using that $3f(g(\xi)+1)<g(\xi+1)$).
We claim that $\theta_0\wedge\neg\theta_1$ is unbounded below $\om_1^X$.
Once we prove the claim, we can then use Lemma \ref{thm: effectively unbounded} to get a model $\A_\si$ of $\theta_0\wedge\neg\theta_1$ of high Scott rank with $\om_1^{\A_\si}=\om_1^X$.
To prove the claim, start by noticing that $\theta_0$ is $\Pii_{g(\xi)+1}$ and is unbounded below $\om_1^X$ as witnessed by $\A_{\si^-}$.
Hence, there is a $\Pii_{f(g(\xi)+1)}$  formula $\psi$ such that both $\theta_0\wedge \psi$ and $\theta_0\wedge\neg\psi$ are unbounded below $\om_1^X$. 
For any model $\B\models \theta_0\wedge\neg\psi$ we have $\A_{\si^-}\not\equiv_{f(g(\xi)+1)} \B$, and hence, by Lemma \ref{le: computable bf}, $\B\not\models \theta_1$.
It follows that since $\theta_0\wedge\neg\psi$ is unbounded below $\om_1^X$, so is $\theta_0\wedge\neg\theta_1$ proving the claim.
Finally, using Lemma \ref{thm: effectively unbounded} again, let $\A_\si$ be a model of $\theta\wedge \neg\theta_1$ of high Scott rank with $\om_1^{\A_\si}=\om_1^X\leq SR(\A_\si)$.

To see that we have built a $g$-proper $\a$-tree consider $\tau,\rho\in \fseq{\a}$, and let $\xi$ be the least with $\tau(\xi)\neq\rho(\xi)$.
Suppose $\tau(\xi)=0$ and $\rho(\xi)=1$.
Let $\si$ be $\rho\upto \xi+1$ followed by 0's, and $\si^-$ be $\tau\upto \xi+1$ followed by 0's.
From the construction we get that $\A_\tau\equiv_{g(\xi+1)}\A_{\si^-} \not\equiv_{g(\xi+1)} \A_\si \equiv_{g(\xi+1)} \A_\rho$ as needed.
\end{proof}

We are now ready to use an overspill argument.

\begin{lemma}(ZFC+PD) \label{le: alpha star tree}
Suppose that there is a $g\colon\om_1\to\om_1$ such that for every $\a$ there is a $g$-proper $\a$-tree of models of $T$.
Then, relative to every oracle $X$ on a cone, there is an $X$-computable, proper, $\a^*$-friendly $\a^*$-tree of models of $T$ for some $\a^*\in \H^X\sminus \om_1^X$.
\end{lemma}
\begin{proof}
By Theorem  \ref{thm: looks computable} there is a cone of oracles  according to which $g$ looks computable.
Using projective Turing determinacy (which follows from PD), all we need to do is show that the set of $X$ satisfying the thesis of the lemma is co-final in the Turing degrees.
So, given $Z$ we need to find $X\geqt Z$ with this property.
Assume that according to $Z$, $g$ looks computable.
By the hypothesis of the lemma, there is an $Y$ which computes a $g$-proper $\a$-tree of models of $T$ for each $\a<\om_1^Z$, which might not be $\a$-friendly.
But it just takes $2\a$-jumps over the model to compute all the $(<\a)$-back-and-forth relations.
Then, if $X$ computes every set hyperarithmetic in $Y$, it computes a $g$-proper $\a$-friendly $\a$-tree of models of $T$ for each $\a<\om_1^Z$.
The set of $X$ which,  for each $\a<\om_1^Z$,  compute a $g$-proper $\a$-friendly $\a$-tree is $\Si^1_1(Z)$, as the quantifier $\forall \a<\om_1^Z$ can be replaced by a second-order $\exists$-quantifier.
Thus, by Gandy's basis theorem, there is such an $X$ with $\om_1^X=\om_1^Z$.
Now, the set of $\b\in\H^X$ such that $X$ computes a  $g$-proper, $\b$-friendly $\b$-tree is $\Si^1_1(X)$, and contains $\om_1^X$.
By an overspill argument, every such $X$ computes an $g$-proper, $\a^*$-friendly $\a^*$-tree for some $\a^*\in\H^X\not\in\om_1^X$, as needed.
\end{proof}

%%%%%%%%%%%%%%%%%%

\subsection{Tying the loose ends} \label{ss: tying}

We can now put all the pieces together and prove Theorem \ref{thm:main}, that every uniformly effectively  dense theory is on top under effective reducibility relative to every oracle on a cone. 

\begin{proof}[Proof of Theorem \ref{thm:main}]
Let $T$ be uniformly effectively dense witnessed by $f$.
By Lemma \ref{le: alpha tree}, we have that for every  $\a\in\om_1$, a $g$-proper $\a$-tree of models of $T$ exists, where $g$ is defined by iterating $f$.
Then, by Lemma \ref{le: alpha star tree}, we have that relative to every oracle $X$ on a cone, there is an $X$-computable proper, $\a^*$-friendly $\a^*$-tree of models of $T$ for some $\a^*\in \H^X\sminus \om_1^X$.
Finally, we apply Theorem \ref{thm: tree to top} to get that $T$ is on top under effective reducibility relative to every such $X$.
\end{proof}

%%%%%%%%%%%%%%%%%%%%%%%%%%%%%%%%%%%%%%%%%%%%%%%%%%%%%%%%%%%%%%%%%%%%%%%%%%%%%%%%%%%%%%%%%%%%%%%%%%%%%%%%%%%%%%%%%%%%%%%%%%%%%%%%%%%%%%%%%%%%%%%%%%%%%%%%%%%%%%%%%%%%%%%%%%%%%%%%%%%%%%%%%%%%%%%%%%%%%%%%%%%%%%%%%%%%%%%%%%%%%%%%%%%%%%%%%%%%%%%%%%%%%%%%%%%%%%%%%%%%%%%%%%%%%%%%%%%%%%%%%%%%%%%%%%%%%%%%%%%%%%%%%%%%%%%%%%%%%%%%%%%
\section{Case Study: Linear orderings}

In this section, we prove that the theory of linear orderings has the no-intermediate-extension property.
This implies that it satisfies Vaught's conjecture  by Theorem \ref{thm: becker}.
The first step in this proof is to show that if we have a computable linear ordering $\L$ of high Scott rank, then we can write it as $\sum_{q\in \QQ}\B_q$ where  each $\B_q$ has high Scott rank.
The following step is to replace each linear ordering $\B_q$ by another $\hat{\B_q}$ that is $\a$-equivalent, to get a linear ordering $\hat{\L}$ that is $\a$-equivalent to $\L$ and has certain desired properties.
What we are using here is the following property.

\begin{lemma}\label{le: equivalent sums}
If for all $i\in \C$ (where $\C$ is a linear orderings) we have linear orderings $\A_i\equiv_\a\B_i$, then $\sum_{i\in \C}\A_i\equiv_\a\sum_{i\in\C}\B_i$.
\end{lemma}
\begin{proof}
Add to the linear orderings $\sum_{i\in \C}\A_i$ and $\sum_{i\in\C}\B_i$ unary relations $U_i$, one for each $i\in \C$, identifying the segment that corresponds to either $\A_i$ or $\B_i$.
It is straightforward to show that these two structures in this new language are $\a$-equivalent (by transfinite induction on $\a$ using the back-and-forth definition of $\equiv_\a$).
But then, forgetting about these new relations, we get that the linear orderings are $\a$-equivalent.
\end{proof}

To get the decomposition of $\L$ as mentioned above, the main idea is to consider the following convex equivalence relation on a linear ordering.

\begin{definition}
Given a linear ordering and an ordinal $\a$, we define a binary relation $\sim_\a$ on $\L$ given by: for $a<b\in \L$ let
\[
a\sim_\a b  \iff SR((a,b)_\L) < \a,
\]
where $(a,b)_\L$ is the open segment $(a,b)$ inside $\L$.
\end{definition}

The idea of considering this equivalence relation is similar to ideas of Kach and Montalb\'an when they were thinking  Vaught's conjecture for Boolean algebras.
That question is still open.
It is also open whether an analog of Lemma \ref{le: SR of omega sum} holds for Boolean algebras.

%%%%%%%%%%%%%%
\subsection{Basic results on Scott ranks of linear orderings}

To prove the basic results about $\sim_\a$, we need a few lemmas that will help us compute the Scott ranks of various linear orderings. 
Most of the bounds in these lemmas are probably not sharp, but are enough for our purposes.

We will repeatedly use the fact that $(\A,a_1,...,a_k)\leq_\xi(\B,b_1,...,b_k)$, where $\A= \A_0+\{a_1\}+\A_1+\{a_2\}+\cdots+\{a_k\}+\A_k$ and $\B_0+\{b_1\}+\B_1+\{b_2\}+\cdots+\{b_k\}+\B_k$, if and only if $\A_i\leq_\xi\B_i$ for each $i\leq k$ (see \cite[Lemma 15.7]{AK00}).
It follows that the $\Pii_\a$-type of a tuple $(a_1,...,a_k)$ in $\A$, is determined by the $\Pii_\a$-theories of the $\A_i$ for $i=0,...,k$.

\begin{lemma}\label{le: SR sum 1}
For two linear orderings $\A$, $\B$, 
\[
\max\{SR(\A),SR(\B)\} \leq SR(\A+1+\B)\leq \max\{SR(\A),SR(\B)\} +3.
\] 
\end{lemma}
\begin{proof}
Let as call $c$ the element in place of the `1' in $\A+1+\B$.
First, to show that $SR(\A) \leq SR(\A+1+\B)$ we observe for $\abar,\bbar\in \A^{<\om}$ and $\a\in\om_1$, we have that $(\A;\abar)\leq_\a(\A;\bbar)$ if and only if $(\A+1+\B;\abar,c)\leq_\a(\A+1+\B;\bbar,c)$.
For each $\a< SR(\A)$ we know that there are tuples $\abar,\bbar\in \A^{<\om}$ such that $\abar\leq_\a\bbar$ but $\abar\not\equiv_{\a+1}\bbar$ within $\A$ (as otherwise $\A$ would have Scott rank $\leq \a$ \cite{AK00}).
But then the same is true for $\abar c$ and $\bbar c$ within $\A+1+\B$, showing that $\a<SR(\A+1+\B)$.

The same way we can show that $SR(\B) \leq SR(\A+1+\B)$.

For the other direction, let $\a=\max\{SR(\A),SR(\B)\}$.
Then, each of $\A$ and $\B$ have a $\Pii_{\a+1}$ Scott sentence, and hence $\A+1+\B$ has a $\Sii_{\a+2}$ Scott sentence saying that there exists an element such that the linear ordering to the left satisfies the Scott sentence for $\A$, and the one to the right the sentence for $\B$.
It then follows that $SR(\A+1+\B)\leq \a+2$.
\end{proof}

\begin{corollary}
If $\a$ is a limit ordinal, then $\sim_{\a}$ is an equivalence relation on any linear ordering $\L$.
\end{corollary}
\begin{proof}
Symmetry and reflexivity are obvious from the definition.
Transitivity follows from the lemma above.
\end{proof}

\begin{lemma}\label{le: SR sum}
For two linear orderings $\A$, $\B$,  
\[
SR(\A+\B)\leq \max\{SR(\A),SR(\B)\} \cdot 2+3.
\] 
\end{lemma}
\begin{proof}
Let $\a=\max\{SR(\A),SR(\B)\}$.

First, suppose that there are some $x\in\A$ and $y\in \B$ such that $\A_{<x}\isom \A+\B_{<y}$.
Via this isomorphism we get a $z\in\A$ such that $(z,x)_\A\isom (x,y)_{\A+\B}$, which by the Lemma \ref{le: SR sum 1} we know has Scott rank $\leq \a$.
But then $\A+\B\isom \A_{<x}+1+(z,x)_\A + 1+ \B_{>y}$, all of which have Scott rank $\leq \a$, and  by Lemma \ref{le: SR sum 1}, $SR(\A+\B)\leq \a+2$.

Suppose now that for no $x\in\A$ and $y\in \B$ is $\A_{<x}\isom \A+\B_{<y}$.
We can then define the $\A$-cut within $\A+\B$ as the set of all $z\in \A+\B$ such that $(\A+\B)_{<z}$ is isomorphic to $\A_{<x}$ for some $x\in \A$.
This is a $\Sii_{\a+1}$ formula using that each $\A_{<x}$ has a $\Pii_{\a+1}$ Scott sentence. 
Now, to define an orbit in $\A+\B$ all we have to do is find its definition within either $\A$ or $\B$, and then relativize this definition to the Scott sentence of either $\A$ or $\B$, getting a $\a\cdot 2+2$ definition.
It follows that $SR(\A+\B)\leq\a\cdot 2+3$.
\end{proof}

The next lemma will become handy.

\begin{lemma}[Lindenbaum~\cite{Rose1982}] \label{thm: Lindenbaun}
If $\X,\Y$ are linear orderings such that $\X$ is isomorphic to an initial segment of $\Y$ and $\Y$ is isomorphic to an end segment of $\X$, then $\X\isom\Y$.
\end{lemma}

\begin{lemma}\label{le: SR of omega sum}
If $SR(\L_{<x})<\a$ for all $x\in \L$, then $SR(\L)\leq \a+4$.
\end{lemma}
\begin{proof}
The proof is divided in two cases.

Case 1: Suppose that for co-finally many $x\in \L$, the set $\{y\in \L: \L_{<y}\isom \L_{<x}\}$ is bounded above in $\L$.
Take $z\in \L$.
We will find a $b\in\L$ such that the following formula defines the automorphism orbit of $z$ within $\L$:
\begin{description}
\item[$\Phi_b(w)$] There exists $v>w$ such that $\L_{<v}\isom \L_{<b}$, and for every $u\geq v$ with $\L_{<u}\isom \L_{<b}$, we have that $\L_{<u}\models \varphi_{z,b}(w)$,
\end{description}
where $\varphi_{z,b}$ is the $\Pii_\a$ formula that defines the orbit of $z$ within $\L_{<b}$.
The formula $\Phi_b(w)$ is $\Sii_{\a+2}$ because checking ``$\L_{<v}\isom \L_{<b}$'' is $\Pii_\a$ (as a formula with one free variable $v$) using that $\L_{<b}$ has a $\Pii_{\a}$ Scott sentence. 
We now need to prove two things:
\begin{enumerate}
\item There is a $b$ such that for every $u\geq b$ with $\L_{<u}\isom \L_{<b}$ we have that $\varphi_{z,u}=\varphi_{z,b}$. \label{step 01}
\item For such $b$, the formula $\Phi_b(w)$ defines the orbit of $z$. \label{step 02}
\end{enumerate}
For  (\ref{step 01}) we chose $x>z$ and $b>z$ such that $\{y\in \L: \L_{<y}\isom \L_{<x}\}$ is bounded by $b\in \L$.
Take $u\geq b$ with $\L_{<u}\isom \L_{<b}$.
The supremum of the set $\{y\in \L: \L_{<y}\isom \L_{<x}\}$ determines a cut in both $\L_{<u}$ and $\L_{<b}$ which is invariant under automorphisms.
The right part of the cut within both $\L_{<u}$ and $\L_{<b}$ must then be isomorphic, and hence there is an isomorphism between $\L_{<u}$ and $\L_{<b}$ leaving the left part of the cut fixed.
This isomorphism leaves $z$ fixed, and hence $\varphi_{z,u}=\varphi_{z,b}$.
To show (\ref{step 02}) we observe in $\Phi_b(w)$ we can now replace $\L_{<u}\models \varphi_{z,b}(w)$ by $\L_{<u}\models \varphi_{z,u}(w)$.
It is clear now that $\Phi_b(z)$ holds, just because $\L_{<u}\models \varphi_{z,u}(z)$ holds for any $u$ and $z$, and hence any $w$ automorphic to $z$ satisfies $\Phi_b(w)$ too.
For the other direction, suppose now $\L\models \Phi_b(w)$ with witness $v$.
Since $\L_{<v}\models \varphi_{z,v}(w)$, there is an automorphisms of $\L_{<v}$ mapping $z$ to $w$.
This automorphism can now be extended to an automorphism of the whole of $\L$ mapping $z$ to $w$.

Case 2. Suppose we are not in any of the previous case.
Furthermore, suppose that for any $a$, $\L_{\geq a}$ does not satisfy the condition of case 1, as otherwise we would have $SR(\L_{>a})\leq \a+2$ and by Lemma \ref{le: SR sum 1} that $SR(\L)\leq \a+4$.

Take $z\in \L$; again, we will find a $b_0\in\L$ such that the formula $\Phi_{b_0}(w)$ defines the automorphism orbit of $z$ in $\L$.
Again, the main step is to find $b_0$ as in condition (\ref{step 01}) above.
The same proof we used for (\ref{step 02}) above would then show that $\Phi_{b_0}(w)$ indeed defines the automorphism orbit of $z$.

Let $a_0>z$ be such that $\{y\in \L: \L_{<y}\isom \L_{<a_0}\}$ is unbounded, which exists because we are not in case 1.
Let $b_0$ be such that $\L_{<b_0}\isom \L_{<a_0}$ and $\{y\in \L: [a_0,y)_\L\isom [a_0,b_0)_\L\}$ is unbounded;  such $b_0$ exists because otherwise, $\L_{\geq a_0}$ would satisfy the condition of case $1$.
Take $a>b_0$ with $\L_{<a}\isom \L_{<b_0}$.
To show that $\varphi_{z,a}=\varphi_{z,b_0}$ we will show that there is an isomorphism between $\L_{<a}$ and $\L_{<b_0}$ that leaves $z$ fixed.
Call $\L_{<a_0}=\A$, $[a_0,b_0)_\L=\B$ and $[a_0,a)_\L=\C$.
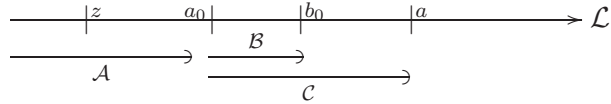
\[
\xymatrix@R=0.1pc{
\ar[rrrrrr] & |^z			&	{}^{a_0}| 	&	|^{b_0}	& |^a	&   & \L \\
\ar@{-)}_{\A}[rr]  & 	&   \ar@{-)}^{\B}[r] & & & & \\
				&	&   \ar@{-)}_{\C}[rr] & & & &
} \]
We will now prove that $\B\isom\C$, and thus that here is an isomorphism between $\A+\B$ and $\A+\C$ fixing $\A$, and hence fixing $z$. 
Clearly $\B$ is an initial segment of $\C$ (because $b_0<a$), but also note that $\C$ is  isomorphic to an initial segment of $\B$ because $\{y\in \L: [a_0,y)_\L\isom \B\}$ is unbounded, and hence there is such a $y>a$.
Via the isomorphism from $\L_{<b_0}$ to  $\L_{<a_0}$, the image of $a_0$ is some $b_1$ such that $[a_0,b_0)_\L\isom [b_1,a_0)_\L \isom\B$.
Via the isomorphism from $\L_{<a}$ to  $\L_{<a_0}$, the image of $a_0$ is some $a_1$ such that $[a_0,a)_\L\isom [a_1,a_0)_\L \isom\C$.
Then, either $a_1\leq b_1$ or $b_1\leq a_1$, so, either $\B$ is a finial segment of $\C$ or $\C$ is a final segment of $\B$.
In either case, by Lemma \ref{thm: Lindenbaun} we get that $\B$ and $\C$, as isomorphic, which is what we needed to get (\ref{step 01}).
\end{proof}

\begin{lemma}\label{le: quotient dense}
For a computable linear ordering $\L$ of high Scott rank, $\L/{\sim_{\om_1^{CK}}}$ is dense.
\end{lemma}
\begin{proof}
Suppose, towards a contradiction,  that $a_0$ and $a_1$ are in adjacent equivalence classes in $\L/{\sim_{\om_1^{CK}}}$.
That means that for every $x\in (a_0,a_1)_\L$, either $SR((a_0,x)_\L)<\om_1^{CK}$ or $SR((x,a_1)_\L)<\om_1^{CK}$.
By $\Si^1_1$ bounding, there is an $\a<\om_1^{CK}$ such that for every $x\in(a_0,a_1)_\L$ either $SR((a_0,x)_\L)<\a$ or $SR((x,a_1)_\L)<\a$.
Write $(a_0,a_1)_\L$ as $\A+\B$ where $\A$ consists of the $x$'s with $SR((a_0,x)_\L)<\a$ and $\B$ of the other ones.
By the previous lemma, we have that $SR(\A)\leq \a+4$, and, applying the previous lemma to $\B^*$, that $SR(\B)\leq \a+5$ too.
By Lemma \ref{le: SR sum}, we then have that $SR((a_0,a_1)_\L)\leq\a\cdot 2+11<\om_1^{CK}$, and thus that $a_0\sim_{\om_1^{CK}} a_1$, contradicting the assumption that they are in different equivalence classes.

To see that we must have more than one equivalence class, consider $1+\L+1$.
Since $SR(\L)\geq\om_1^{CK}$, the 1s at the extremes are not $\sim_{\om_1^{CK}}$-equivalent.
So $1+\L+1/{\sim_{\om_1^{CK}}}$ has more than one element and is dense by the previous paragraph.
So $\L/{\sim_{\om_1^{CK}}}$  must also have more than one equivalence class.
\end{proof}

%%%%%%%%%%%%%%%%%%%
\subsection{Vaught's conjecture for Linear orderings} \label{ss: VC LO}
In this section, we give a new proof of Rubin's theorem that the theory of Linear orderings satisfies Vaught's conjecture (in the sense that all extensions do) \cite{Rub74, Ste78}.

Consider a infinitary sentence $T$ in the language $\{\leq\}$, which extends the theory of linear orderings.
We may assume $T$ is given by a computably infinitary sentence, as we can always relativize the rest of the proof later.
If there is a bound on the Scott ranks of the models of $T$, we then have that the isomorphism problem among the reals coding models of $T$ is Borel (see \cite[Theorem 12.2.4]{GaoBook}).
Then, we can apply Silver's theorem \cite{Sil80}, which says that every Borel equivalence relation has either countably or continuum many equivalence classes, to get that $T$ has either countably or continuum many models.

Thus, let us assume that $T$ is unbounded, and hence it has a model $\L$ with $\om_1^{CK}=\om_1^L\leq SR(\L)$ (by Lemma \ref{thm: effectively unbounded}).
Let us also assume that $T$ has less than continuum many models.
Relativizing again, let us assume that $\L$ has a computable copy.

Let $\a$ be a limit ordinal be such that $T$ is $\Pic_\a$.
Take any countable linear ordering $\A$.
We will show that there is a linear ordering $\Lhat\equiv_\a\L$ such that $\Lhat/{\sim_{\om_1^{CK}}}\isom \A$, showing that there are continuum many models of $T$.
(Notice that if $\Lhat\equiv_\a\L$, then $\Lhat\models T$.)

By Lemma \ref{le: quotient dense}, $\L/{\sim_{\om_1^{CK}}}$ is dense.
By using an isomorphism between $\QQ$ and $\QQ\cdot\ZZ\cdot \A$, we can write 
\[
\L = \sum_{q\in\A}(\sum_{n\in\ZZ} \B_{q,n}),
\]
where each $\B_{q,n}$ is such that  $\B_{q,n}/{\sim_{\om_1^{CK}}}$ is still dense and hence  $SR(\B_{q,n})\geq\om_1^{CK}$.
Let $\a<\a_0<a_1<\a_2<....$ be a sequence of limit ordinals with limit $\om_1^{CK}$.
For each $q\in\A$ and $n\in\ZZ$, let $\Bhat_{q,n}$ be such that $\Bhat_{q,n}\equiv_{\a_{|n|}}\B_{q,n}$ and $\a_{|n|}\leq SR(\Bhat_{q,n})\leq \a_{|n|}+1$.
To build such a linear ordering $\Bhat_{q,n}$ one needs to construct a model of the $\Pic_{<\a_{|n|}}$-theory of $\B_{q,n}$, but omitting all the non-principal $\Pic_{<\a_{|n|}}$-types  that are realized in any model of $T$ (for the type omitting theorem see \cite[Theorem III.3.8]{Bar75}).
That there are only countably many such types follows from the fact that otherwise there would be continuum many, and hence there would be continuum many models of $T$, which we are assuming there are not.

Let $\Lhat = \sum_{q\in\A}(\sum_{n\in\ZZ} \Bhat_{q,n})$.
By Lemma \ref{le: equivalent sums} we then have that $\L\equiv_\a\Lhat$.
It is not hard to see that if we are given $b\in \Bhat_{q,n}$ and $c\in \Bhat_{p.m}$, then $b\sim_{\om_1^{CK}} c$ if and only if $p=q$.
So, for each $q\in\A$ we have that $(\sum_{n\in\ZZ} \Bhat_{q,n})$ is a single $\sim_{\om_1^{CK}}$ equivalence class and $\L/{\sim_{\om_1^{CK}}}\isom\A$.

%%%%%%%%%%%%%%%%%%%%
\subsection{Linear orderings are uniformly effectively dense}
We now give a proof of Theorem \ref{thm: linear no int} that the theory of linear orderings has the no-intermediate-extension property.
To prove it we use Theorem \ref{thm:main} and the following theorem.

\begin{theorem}
The theory of linear orderings is uniformly effectively dense.
\end{theorem}
\begin{proof}
Consider again a $\Pii_\a$ sentence $T$ in the language $\{\leq\}$, which extends the theory of linear orderings, and which is effectively unbounded.
We will prove that there is a $\Pii_{\a...}$ sentence $\psi$ such that both $T\wedge \psi$ and $T\wedge\neg\psi$ are unbounded below $\om_1^{CK}$.
We may assume $T$ is given by a computably infinitary sentence, and, as above, assume $T$ has a computable model $\L$ of high Scott rank, as we can relativize the proof later.
Also, using that $\L/{\sim_{\om_1^{CK}}}$ is dense as above, we find a decomposition
\[\L = (1+\A_0+1+\A_1+1+\A_2+1+\cdots)+(\cdots +\B_2+\B_1+\B_0),
\]
were $\A_i$ and  $\B_i$ have Scott rank at least $\om_1^{CK}$ for all $i$.
Let us assume $\a$ is a limit ordinal; if not consider $\a+\om$ instead.
We will define a $\Pii_{\a+\om\cdot 2}$ sentence extending $T$, false about $\L$ and that is unbounded below $\om_1^{CK}$.
We consider three cases.

Case 1: Suppose that for some $i$, $\A_i$ is not $\a$-equivalent to any linear ordering of Scott rank $\a$.
We will now build $\Ahat\equiv_\a\A_i$ which satisfies some type which is not realized in $\L$.
By Lemma \ref{le: computable bf}, we have that if a structure $\Ahat_i\models \Pic_{<\a}$-theory$(\A)$, then $\Ahat_i\equiv_\a\A_i$.
We will define $\Ahat_i$ using the type-omitting theorem (see for instance \cite[Theorem III.3.8]{Bar75}) as a model of $\Pic_{<\a}$-theory$(\A)$ which omits the following countable list of non-principal types:
For each $\Pic_{<\a}$-type $\Phi(x_0,...,x_k)$ we will define a $\Pic_{<\a}$-type $\Phihat(x_1,...,x_{k-1})$ obtained by essentially forgetting about what happens to the left of $x_0$ and to the right of $x_k$.
In other words, given a $\Pic_{<\a}$-type $\Phi(x_0,...,x_k)$ realized in $\L$ by $a_0<a_1<\cdots<a_k$, we let  $\Phihat(x_1,...,x_{k-1})$ be the $\Pii_{<\a}$-type of $a_1,...,a_{k-1}$ within the linear ordering $(a_0,a_k)_\L$.
The list of types to omit consists of all the non-principal $\Pic_{<\a}$-types $\Phihat(x_1,...,x_{k-1})$ that come from a $\Pic_{<\a}$-type $\Phi(x_0, x_1,...,x_{k-1},x_k)$ realized in $\L$.
Let $\Lhat$ be defined by replacing $\A_i$ by $\Ahat_i$, and leaving the rest of $\L$ untouched.
Since the rest of $\L$ has Scott rank at least $\om_1^{CK}$, so does $\Lhat$.
By our assumption, $\Ahat_i$ does not have Scott rank $\a$, and hence there is some tuple $a_1<\cdots a_{k-1}\in\A_i$ satisfying some non-principal $\Pic_{<\a}$ type $\hat{\Gamma}$.
Let $\Gamma$ be the $\Pic_{<\a}$-type within $\Lhat$ of $a_0,a_1,...,a_k$ (where $a_0$ and $a_k$ are the 1's surrounding $\A_i$).
This type is not realized in $\L$ because it would have been omitted in $\Ahat_i$ otherwise.
The $\Sii_{\a+1}$ formula saying that there is a tuple in $\L$ satisfying $\Gamma$ is true in $\Lhat$ but not in $\L$.
Since $T$ and this formula are true in $\Lhat$, it is unbounded below $\om_1^{CK}$.

Case 2. Suppose now that, for some $i$, $\A_i$ is not $\a$-equivalent to any linear ordering of Scott rank $\a+\om$.
The proof is the same as above. The separating formula is now $\Sii_{\a+\om+1}$.

Case 3. None of the previous cases hold.
Let $\Lhat$ be built by replacing each $\A_i$ by an $\a$-equivalent linear ordering of Scott rank $\a$.
Let $\Ltil$ be built by replacing each $\A_i$ by an $\a$-equivalent linear ordering of Scott rank $\a+\om$.
Both have Scott rank at least $\om_1^{CK}$ because $\sum_{i\in\om^*}\B_i$ does.
The $\Sii_{<\a+\om\cdot 2}$ formula that says that there is some $x$ such that $SR(\L_{<x})=\a+\om$ is true in $\Ltil$ and not in $\Lhat$, both being models of $T$.
So, both, $T$, together with this formula and together with its negation, are both unbounded below $\om_1^{CK}$.
\end{proof}

\newcommand{\etalchar}[1]{$^{#1}$}

%
%
%\bibliography{bftypes}
%\bibliographystyle{alpha}
%

\end{document}